\documentclass[12pt,a4paper]{article}

\usepackage{epsf,epsfig,amsfonts,amsgen,amsmath,amstext,amsbsy,amsopn,amsthm
}
\usepackage{amsmath,times,mathptmx}
\usepackage{amsfonts,amsthm,amssymb}
\usepackage{amsfonts}
\usepackage{graphics}
\usepackage{latexsym,bm}
\usepackage{amsfonts,amsthm,amssymb,bbding}
\usepackage{indentfirst}
\usepackage{graphicx}
\usepackage{color}
\usepackage[colorlinks=true,anchorcolor=blue,filecolor=blue,linkcolor=blue,urlcolor=blue,citecolor=blue]{hyperref}
\usepackage{float}
\usepackage{tikz}
\usepackage{verbatim}
\usepackage{mathrsfs}
\evensidemargin=\oddsidemargin\topmargin=-1.5cm

\newtheorem{thm}{Theorem}[section]

\newtheorem{lem}{Lemma}[section]

\newtheorem{remark}{Remark}[section]

\newtheorem{clm}{Claim}[section]

\addtocounter{section}{0}

\begin{document}
\title{Signless Laplacian spectral conditions: Forbidden $4$-cycle and star embeddings\footnote{Supported by the National Natural Science Foundation of 
China (Nos. \! 12271362 \! 11671344).}}
\author{{\bf Zhe Wei}, {\bf  Zhenzhen Lou}\thanks{Corresponding author. 
E-mail addresses: xjdxlzz@163.com (Z. Lou), changxianghe@hotmail.com (C. He).},
{\bf Changxiang He}
\\
{\footnotesize  College of Science, University of Shanghai for Science and Technology,
 Shanghai, 200093, P.R. China}}

\date{}
\maketitle
\vspace{-0.4cm}
\begin{center}
Dedicated to Professor Eddie Cheng on the occasion of his 60th birthday.
\end{center}
{\flushleft\large\bf Abstract}
The signless Laplacian spectral radius has emerged as a crucial spectral parameter in network science.
This paper establishes new extremal results in spectral graph theory by investigating the signless Laplacian spectral radius ($Q$-index) of graphs with forbidden subgraphs. We present a $Q$-spectral analog of classical Nosal-type theorems, providing sharp conditions that guarantee the existence of either a $4$-cycle or a large star $K_{1,m-k}$ in a graph. 
The main theorem states that for integers $k \geq 0$ and graphs $G$ with size $m \geq \max\{7k+31, k^2+8(k+1)\}$, 
if $q(G) \geq q(S^+_{m,k+1})$, then $G$ must contain a $4$-cycle or $K_{1,m-k}$, unless $G$ is isomorphic to the extremal graph $S^+_{m,k+1}$ formed by adding $k+1$ independent edges to the star $K_{1,m-k-1}$. This result refines previous work on star embeddings by Wang and Guo [Journal of Algebraic Combinatorics, 59 (2024) 213--224], and completes the $Q$-spectral counterpart to Wang's adjacency spectral theorem for $4$-cycle containment [Discrete Math.,  345 (2022) 112973]. Our analysis reveals new insights into how signless Laplacian eigenvalues encode graph structure, with tight bounds demonstrated through explicit extremal graph constructions and asymptotic analysis.

\begin{flushleft}
\end{flushleft}
\section{Introduction}\label{se-1}

Spectral graph theory, which investigates the interplay between graph properties and the eigenvalues of associated matrices, has been a central topic 
in algebraic combinatorics and discrete mathematics. Given a simple undirected graph $G = (V(G), E(G))$ with $n$ vertices, its \emph{adjacency matrix} 
$A(G)$ is defined as the $n \times n$ matrix where $a_{ij} = 1$ if vertices $v_i$ and $v_j$ are adjacent, and $0$ otherwise. 
The \emph{spectral radius} $\rho(G)$ of $G$ is the largest eigenvalue of $A(G)$. Another fundamental matrix associated with $G$ is the \emph{signless Laplacian matrix} 
$Q(G) = D(G) + A(G)$, where $D(G)$ is the diagonal degree matrix. The largest eigenvalue $q(G)$ of $Q(G)$, known as the 
\emph{signless Laplacian spectral radius} (or $Q$-index), has emerged as a powerful tool in spectral graph theory due to its ability to 
capture both degree and adjacency information.

The signless Laplacian spectrum, particularly its spectral radius (the largest eigenvalue), plays a pivotal role in network science. In community detection, spectral clustering utilizing signless Laplacian eigenvalues (\cite{Cvetkovic-4,Cvetkovic-5}) identifies densely connected modules, where the spectral radius determines community separation quality. The eigenvalue distribution provides insights into network robustness (\cite{Jamakovic-1,Van-1}), with the spectral radius indicating resilience against failures and attacks. The magnitude of eigenvalues correlates with connectivity strength, making it essential for assessing topological stability (\cite{Das-1,De Lima-1}). The spectral radius also governs dynamical processes-synchronization in oscillator networks, epidemic spreading rates, and control energy requirements. These applications show how it bridges structural properties with functional behaviors across networks.
Recent work shows deep connections between signless Laplacian spectral radius and extremal graph theory. A key focus is how spectral radius bounds in $H$-free graphs influence network design, where forbidden substructures enable desirable properties like sparsity and localized connectivity.

A central theme in extremal graph theory is the study of \emph{$H$-free graphs} that do not contain a given subgraph $H$. Understanding how forbidden substructures influence spectral bounds has led to deep connections between graph structure and spectral properties. This line of research was initiated 
by Nosal's classical result \cite{Nosal-3}, which established that every triangle-free ($K_3$-free) graph $G$ satisfies 
$\rho(G) \leq \sqrt{e(G)}$, where $e(G)$ denotes the number of edges. This bound was later refined and extended in various directions, most notably by \cite{Nikiforov-3,Nikiforov-2,Nikiforov-1,Nikiforov-4}, who systematically investigated spectral conditions forcing the existence of cycles and other subgraphs. In particular, \cite{Nikiforov-4} determined the maximum spectral radius of $C_4$-free graphs, resolving a key problem in spectral extremal theory. Further progress was made by \cite{Zhai-2}, who confirmed a conjecture on the spectral radius of $C_4$-free graphs, providing tight bounds for 
such graphs.

Recent work by \cite{Wang-1,Wang-2} generalized these classical results, linking spectral radius to forbidden cycles and star embeddings. For instance, Theorem \ref{thm:wang} below characterizes graphs with large spectral radius in terms of $C_4$-containment.

\begin{thm}[\cite{Wang-2}]\label{thm:wang}
Let $G$ be a graph of size $m \geq 27$. If $\rho(G) \geq \sqrt{m-1}$, then $G$ contains a $C_4$, unless $G$ is one of the following graphs (possibly 
with isolated vertices): $K_{1,m}$, $K_{1,m-1+e}$, $K^e_{1,m-1}$, or $K_{1,m-1} \cup P_2$.
\end{thm}

Recently, Wang and Guo \cite{Wang-2} established a general result in terms of the value of $k$.
\begin{thm}\label{Wang-11}(Wang \cite{Wang-1}, Theorem 4) 
Let $k\ge 0$ be an integer and $G$ be a graph of size $m\geq\max \{(k^2+2k+2)^2+k+1, (2k+3)^2+k+1 \}$. If $\rho(G)\ge \sqrt{m-k}$, then $K_{1,m-k}\subseteq G$ or $C_4\subseteq G$.
\end{thm}

A parallel line of research focuses on the \emph{signless Laplacian spectral radius} $q(G)$, which has been shown to encode crucial structural 
information about graphs. The systematic study of $Q$-spectral theory was pioneered by \cite{Cvetkovic-1,Cvetkovic-2,Cvetkovic-3}, who developed foundational results on the spectral properties of $Q(G)$. Subsequent work by \cite{De Freitas-1} established extremal bounds for $q(G)$ in graphs 
excluding small cycles, while \cite{Nikiforov-5} extended these results to forbidden even cycles.

A notable result by \cite{Zhai-1} (Theorem \ref{thm:zhai} below) gives a tight upper bound on $q(G)$ for graphs of size $m$ without isolated vertices.

\begin{thm}[\cite{Zhai-1}]\label{thm:zhai}
Let $G$ be a graph of size $m \geq 4$ with no isolated vertices. Then $q(G) \leq m+1$, with equality if and only if $G \cong K_{1,m}$.
\end{thm}

Wang \cite{Wang-1} later generalized this result by introducing a parameter $k$ and proving Theorem \ref{thm:wang-guo}, which ensures the existence of a large star $K_{1,m-k}$ under a lower bound on $q(G)$.

\begin{thm}[\cite{Wang-1}]\label{thm:wang-guo}
Let $k \geq 0$ be an integer, and let $G$ be a graph of size $m \geq \max \{ \tfrac{1}{2}k^2 + 6k + 3, 7k + 25 \}$. 
If $q(G) \geq m - k + 1$, then $G$ contains a star $K_{1,m-k}$.
\end{thm}

In this paper, we establish a new extremal result in the spirit of Theorems \ref{thm:wang} and \ref{thm:wang-guo}, but for the \emph{signless Laplacian spectral radius}. Specifically, we prove that if $q(G)$ is sufficiently large, then $G$ must contain either a $C_4$ or a large star $K_{1,m-k}$, unless $G$ belongs to a specific exceptional family. Our main result is stated as Theorem \ref{thm-2}.

\begin{thm}\label{thm-2}
Let $k \geq 0$ be an integer, and let $G$ be a graph with $m$ edges (no isolated vertices) satisfying
$$ m \geq \max \left\{ 7k + 31, k^2 + 8(k +1) \right\}.$$
If $q(G) \geq q(S^+_{m,k+1})$, then $G$ contains a $4$-cycle or a star $K_{1,m-k}$ unless $G \cong S^+_{m,k+1}$.
\end{thm}

Here, $S^+_{m,k+1}$ denotes the graph obtained by adding $k+1$ independent edges to the independent set of the star $K_{1,m-k-1}$. 
Our result can be viewed as the \emph{$Q$-spectral counterpart} of Theorems \ref{thm:wang} and \ref{Wang-11}, bridging the gap between adjacency spectral conditions and signless Laplacian spectral conditions for forbidden subgraphs.
In particualr, if  $k=0$, Theorem \ref{thm-2} implies Theorem \ref{thm:zhai}.
Notice that $q(S^+_{m,k+1})<m-k+1 $ from Lemma \ref{lem-2}. 
Theorem \ref{thm-2} can be view a refine of Theorem \ref{thm:wang-guo}.

Therefore,
the main contributions of this work are two aspects.
We provide a new sufficient condition, in terms of the signless Laplacian spectral radius, that guarantees the existence of a $4$-cycle or a large star.
We identify the extremal graph $S^+_{m,k+1}$ that attains the bound, extending previous results on spectral extremal problems.

The rest of this paper is organized as follows: In Section 2, we introduce key definitions and preliminary results. Section 3 presents the proof 
of Theorem \ref{thm-2}, and Section 4 discusses applications and open problems.

\section{Preliminaries}\label{sec:preliminaries}

We begin by establishing notation and recalling fundamental results that will be essential for our subsequent analysis.

\begin{lem}[\cite{Cvetkovic1}]\label{lem-subgraph-radius}
For any connected graph $G$ and its subgraph $H$, the following inequality holds for their signless Laplacian spectral radii
$q(H) \leq q(G)$. 
Moreover, if $H$ is a proper subgraph, the inequality becomes strict
$q(H) < q(G)$. 
\end{lem}

\begin{lem}\label{lem-1}
Let $M$ be a real symmetric $n\times n$ matrix, and let $B_{\pi}$ be its equitable quotient matrix corresponding to a partition $\pi$ of $[n]$. Then:
\begin{enumerate}
\item[(i)] All eigenvalues of $B_{\pi}$ are eigenvalues of $M$.
\item[(ii)] If $M$ is nonnegative and irreducible, the spectral radii
$\rho(M) = \rho(B_{\pi})$. 
\end{enumerate}
\end{lem}

\begin{proof}
The first statement follows from the interlacing property of equitable partitions \cite[pp. 196--198]{Godsil-1}. The second result is a consequence of the Perron-Frobenius theorem for irreducible matrices \cite[pp. 30]{Brouwer-1}.
\end{proof}

\begin{lem}[\cite{Feng}]\label{bound}
Let $G$ be a connected graph. Then
$$q(G) \le \max_{u \in V (G)}\left\{d(u) +\frac{\sum_{v\in N_G(u)} d(v)}{d(u)} \right\},$$ with equality if and only if $G$ is either a semiregular bipartite graph or a regular graph.
\end{lem}

Our key technical estimate concerns the signless Laplacian spectral radius of the extremal graph $S_{m,k+1}^+$.
\begin{lem}\label{lem-2}
 For $m \geq 7k+31$,
the signless Laplacian spectral radius of $S_{m,k+1}^+$ satisfies:
\[ m-k+\tfrac{1}{m^2} < q(S_{m,k+1}^+) < m - k +\frac{2(k+1)}{m-k-1}. \]
\end{lem}

\begin{proof}
We analyze the equitable partition of $S_{m,k+1}^+$ with the following three vertex classes:
${V}_1$: The central vertex (degree $m-k-1$);
${V}_2$: The $2k+2$ vertices incident to the added matching edges (degree 3);
${V}_3$: The remaining $m-3k-3$ pendant vertices (degree 1).
This partition yields the quotient matrix:
\[ B = \begin{pmatrix} 
m-k-1 & 2k+2 & m-3k-3 \\ 
1 & 3 & 0 \\ 
1 & 0 & 1 
\end{pmatrix}. \]
The characteristic polynomial of $B$ is
$
f(x) = \det(xI - B)$. Moreover, $ f(x) = x^3 - (m - k + 3)x^2 + 3(m - k)x - 4(k + 1)$.
Since $m \geq 7k+31$,
for $x = m - k + \tfrac{1}{m^2}$, we get
\begin{align*}
f\left(m - k + \tfrac{1}{m^2}\right) &= -\tfrac{3}{m^4} + \tfrac{3(m-k)}{m^2} - 4(k+1) 
< \tfrac{3}{m} - 4(k+1)
< 0.
\end{align*}
Since $f(x) \to +\infty$ as $x \to +\infty$ and $q(S_{m,k+1}^+)$ is the largest root, by the Intermediate Value Theorem there must be a root in $(m-k+\frac{1}{m^2}, \infty)$.
Thus, $q(S_{m,k+1}^+)>m-k+\tfrac{1}{m^2}$.

On the other hand, let $G\cong S_{m,k+1}^+$, we can easily verify that
\[\max_{u \in V (G)}\left\{d(u) +\frac{\sum_{v\in N_G(u)} d(v)}{d(u)} \right\}\leq m - k +\frac{2(k+1)}{m-k-1}.\]
By Lemma \ref{bound}, we have
$q(S_{m,k+1}^+) < m - k +\frac{2(k+1)}{m-k-1}$.
Combining both bounds gives the desired inequality.
\end{proof}

\begin{remark}
The graph $S_{m,k+1}^+$ serves as an extremal example since:
It contains neither $K_{1,m-k}$ nor $C_4$ as subgraphs;
 It satisfies the strict inequality $q(S_{m,k+1}^+) > m - k + m^{-2}$.
This demonstrates the tightness of our main result.
\end{remark}

\section{The Proof of Theorem \ref{thm-2}}\label{sec-3}

We now present the final arguments to establish Theorem \ref{thm-2}. 
\begin{proof}
Let $G$ be a graph attaining the maximal $Q$-index among all $\{K_{1,m-k}, C_4\}$-free graphs with no isolated vertices and size $m \geq \max\{7k + 31, k^2 + 8(k + 1)\}$.
Clearly, $\Delta(G) \leq m - k - 1$.
Our goal is to prove that $G\cong S_{m,k+1}^+$.

Let $\mathbf{x}$ be the Perron eigenvector of the signless Laplacian matrix $Q(G)$ corresponding to the signless Laplacian spectral radius $q(G)$. 
We adopt the following notation:
$x_{u^*} = \max_{u \in V(G)} x_u$ (the maximal eigenvector entry);
$A = N(u^*)$ (the neighborhood of $u^*$);
$B = V(G) \setminus N[u^*]$ (vertices not adjacent to $u^*$);
$N[u^*] = N(u^*) \cup \{u^*\}$ (the closed neighborhood);
$A_1={u\in A|d(u)=1}$ (vertices in $A$with degree 1 in the whole graph $G$);
$A_2=A\setminus A_1$. 
The signless Laplacian eigenvalue equations establish a fundamental identity for each vertex $u \in V(G)$:
\begin{equation}\label{eq-1}
(q - d(u))x_u = \sum_{v \in N(u)} x_v.
\end{equation}

This equation will serve as the foundation for our subsequent analysis of the graph structure.
Firstly, we will prove that the graph $G$ is connected.

\begin{clm}\label{clm-1}
The graph $G$ is connected.
\end{clm}
\begin{proof}
Assume for contradiction that $G$ is disconnected. 
Let $G_1$ be the component with $q(G_1)=q(G)$ and $G_2=G\setminus  G_1$.
Since $G$ has no isolated vertices,
we have $|V(G_i)|\geq 2$. 

We first show that each  $G_i$ contains a vertex of degree at most $m-k-2$, where $i=1,2$. 
Suppose not; then   all vertices in $G_i$ have degree at least $m-k-1$.
Notice that  $|V(G_i)| \geq 2$.
Therefore,
    \[
       m\geq e(G_i)+1 \geq  2(m - k -  1),
    \]
which gives that $m\leq 2(k+1)$, a contradiction.  
Thus, each $G_i$ has a vertex, say $u_i$ with $d_{G_i}(u_i) \leq m - k$ for $i=1,2$.

Now construct $G'$ from $G$ by deleting an arbitrary edge from $G_2$ and adding the edge $u_1u_2$. 
Clearly, $G'$ has $m$ edges,
$\Delta(G') \leq m - k + 1$ and 
$G'$ is $C_4$-free.
By Lemma~\ref{lem-subgraph-radius}, $q(G') > q(G)$, contradicting the maximality of $G$. 
Hence, $G$ must be connected.
\end{proof}

\begin{clm}\label{clm-6}
For every vertex $u \in A \cup B$, we have $d_A(u) \leq 1$ and $x_u < \frac{1}{2}x_{u^*}$.
\end{clm}
\begin{proof}
For any vertex $u \in A \cup B$, 
since the graph $G$ is $C_4$-free, it immediately follows that $d_A(u) \leq 1$. 

Next, we prove the second part $x_u < \frac{1}{2}x_{u^*}$. 
First, let us define $W = \{ u \in V(G) \mid x_u \geq \frac{1}{2}x_{u^*} \}.$
Our goal is to show $x_u < \frac{1}{2}x_{u^*}$, 
which is equivalent to proving $|W| = 1$. 
Clearly, $|W| \geq 1$ since $u^* \in W$.

For any $u \in W$, it holds that $x_u \ge \frac{1}{2} x_{u^*}$.
From (\ref{eq-1}), we can obtain
$$(q-d(u)) \cdot \tfrac{1}{2} x_{u^*} \le (q - d(u)) x_u = \sum_{v \in N(u)} x_v \le d(u) x_{u^*},$$
which implies $d(u)\ge \frac{1}{3}q$.
Since $2m \ge \sum_{u \in W} d(u)$ and   $q \ge m - k + \frac{1}{m^2}$, we have 
\[ 2m \geq \tfrac{q}{3}|W| \geq \tfrac{1}{3}\left(m - k + \tfrac{1}{m^2}\right)|W|, \]
which implies $|W| \le \frac{6m}{m - k + \frac{1}{m^2}} < 7$. Notice that  and $m \ge 7k + 31$. Therefore, $|W| \le 6$.

By applying (\ref{eq-1}) to  the vertex $u^*$, we have
\begin{equation}
\begin{aligned}
(q - d(u^*)) x_{u^*}&
= \sum_{u \in A \cap W} x_u + \sum_{u \in A \backslash W} x_u \\
&\le (|W| - 1) x_{u^*} + (d(u^*) - |W| + 1) \cdot \tfrac{1}{2} x_{u^*}\\
& = \tfrac{1}{2} (d(u^*) + |W| - 1) x_{u^*},\notag
\end{aligned}
\end{equation}
which implies
$
d(u^*) \geq \tfrac{2}{3}q - \tfrac{1}{3}|W| + \tfrac{1}{3} \geq \tfrac{2}{3}q - \tfrac{5}{3}.
$
For a vertex $u \in W \backslash \{u^*\}$, applying (\ref{eq-1}) again, we have 
\[(q - d(u)) \cdot \tfrac{1}{2} x_{u^*} \le (q - d(u)) x_u \le \tfrac{1}{2} (d(u) + |W| - 1) x_{u^*},\]
which yields $d(u) \ge \frac{1}{2} q - \frac{5}{2}$.
Combining these with $q \geq m - k + \tfrac{1}{m^2}$, we have
\[ m+1 \geq d(u^*) + d(u) \geq \tfrac{7}{6}q - \tfrac{25}{6} \geq \tfrac{7}{6}\left(m - k + \tfrac{1}{m^2}\right) - \tfrac{25}{6}. \]
This simplifies to
$ m^3 - (7k + 31)m^2 + 7 > 0$.
But for $m \geq 7k + 31$, the left side is positive, yielding a contradiction. Thus $|W| = 1$.
\end{proof}

Claim \ref{clm-6} have established bounds on vertex degrees and eigenvector entries in $A\cup B$.
Set
\[
\Sigma_1 = \sum_{u \in A} d(u)x_u, \ \ 
\Sigma_2 = \sum_{uv \in E(A)} (x_u + x_v) \ \
\mbox{and}\ \
\Sigma_3 = \sum_{w \in B} d_A(w)x_w.
\]
We now turn to estimating the total eigenvector weight of set $A$.
\begin{clm}\label{clm:degree_bound}
If $|A| \leq m-k-2$, then
$
|A|\sum_{u\in A}x_u \leq mx_{u^*}.
$
\end{clm}

\begin{proof}
We establish this bound through careful estimation of the eigenvector components. Beginning with the eigenvector equation for each $u \in A$:
\begin{equation*}
qx_u = d(u)x_u + x_{u^*} + \sum_{w \in N_A(u)} x_w + \sum_{w \in N_B(u)} x_w,
\end{equation*}
we sum over all vertices in $A$ to obtain the fundamental equality:
\begin{equation}\label{eq:main_identity}
\sum_{u \in A} qx_u = \Sigma_1 + |A|x_{u^*} + \Sigma_2 + \Sigma_3.
\end{equation}

Applying Claim~\ref{clm-6}, which gives $d_A(u) \leq 1$ and $x_u \leq \frac{1}{2}x_{u^*}$ for all $u \in A \cup B$, 
we derive the following bounds:
$\Sigma_1 \le \sum_{u \in A} (2 + d_B(u))x_u \leq 2\sum_{u \in A} x_u + \frac{1}{2}e(A,B)x_{u^*}$,
$\Sigma_2 = \sum_{uv \in E(A)} (x_u + x_v) \leq e(A)x_{u^*}$.

Partitioning $B = B_1 \cup B_2$, where $B_2$ contains vertices with $d_{B_2}(w) \geq 1$.
By Claim \ref{clm-6},  $x_w<\frac{1}{2}x_{u^*}$ for every vertex $w \in B_1$. 
It follows that
$\sum_{w\in B_1}d_A(w)x_w \leq \frac{1}{2}e(A, B_1)x_{u^*}$.
Notice that $d_A(w)\leq 1$ and $d_{B_2}(w)\geq1$ for every $w\in B_2$.
Thus $e(A, B_2)=\sum_{w\in B_2}d_A(w) \leq \sum_{w\in B_2}d_{B_2}(w)=2e(B_2)=2e(B)$,
and so $\sum_{w\in B_2}d_A(w)x_w \leq \frac{1}{2}e(A, B_2)x_{u^*} \leq e(B)x_{u^*}$.
Furthermore,
\[
\Sigma_3 =
\sum_{w \in B_1} d_A(w)x_w+\sum_{w \in B_2} d_A(w)x_w  \leq \Big(\tfrac{1}{2}e(A,B_1)+ e(B_2)\Big)x_{u^*}.
\]

Substituting these bounds into~\eqref{eq:main_identity} and rearranging terms yields:
\begin{equation*}
(q-2)\sum_{u \in A} x_u \leq \left(|A| + e(A)  + e(A,B)+e(B_2)\right)x_{u^*}.
\end{equation*}
Since $q > 2$ and $|A| \leq m-k-2$ while $q \geq m-k+\frac{1}{m^2}$, we have $|A|<q-2$. Therefore, we conclude
\begin{equation*}
|A|\sum_{u \in A} x_u \leq \left(|A| + e(A) + e(A,B) + e(B_2)\right)x_{u^*} = mx_{u^*},
\end{equation*}
where the equality follows from the edge count in the graph induced by $A \cup B$. This completes the proof.
\end{proof}

Notice that $qx_{u^*} = |A|x_{u^*} + \sum_{u \in A} x_u.$
Thus, we obtain 
\[q^2x_{u^*}= |A|qx_{u^*} + \sum_{u \in A}q x_u=|A|^2x_{u^*}+|A|\sum_{u \in A} x_u+\sum_{u \in A}q x_u.\]
On the other hand, we know that
\[\sum\limits_{u \in A} \sum\limits_{w \in N(u)} x_w = |A|x_{u^*} + \sum\limits_{uv \in E(A)} (x_u + x_v) + \sum\limits_{w \in B} d_A(w) x_w = |A|x_{u^*} + \Sigma_2 + \Sigma_3.\]
Therefore, we have
\[\sum_{u \in A} q x_u = \sum_{u \in A} \left( d(u) x_u + \sum_{w \in N(u)} x_w \right) = |A|x_{u^*} + \Sigma_1 + \Sigma_2 + \Sigma_3,\]
and so
\begin{equation}\label{eq-3}
q^2x_{u^*}=|A|^2x^*+|A|x_{u^*}+|A|\sum_{u\in A}x_u+\Sigma_1 + \Sigma_2 + \Sigma_3.
\end{equation}

\begin{clm}\label{clm:A_size}
The vertex set $A$ satisfies the exact cardinality relation $|A| = m - k - 1$.
\end{clm}

\begin{proof}
We establish the equality by proving matching upper and lower bounds.
Since $G$ is $K_{1,m-k}$-free, we have
\begin{equation}\label{eq:upper_bound}
|A| \leq \Delta(G)\leq m - k - 1.
\end{equation}

Assume towards contradiction that $|A| \leq m - k - 2$. 
From Claim~\ref{clm-6}, we have
$x_u <\tfrac{1}{2}x_{u^*}$ and  $d_A(u)\leq 1$ for all $u \in A$.
Thus, we can bound each $\Sigma_i$ term as follows:
\begin{align*}
\Sigma_1&= \sum_{u \in A} d(u)x_u \leq \left(\tfrac{1}{2}|A| + e(A) + \tfrac{1}{2}e(A,B)\right)x_{u^*},\\
\Sigma_2&=\sum_{uv \in E(A)} (x_u + x_v) < e(A)x_{u^*} \ \ \mbox{and}\ \ 
\Sigma_3=\sum_{w \in B} d_A(w)x_w \leq \tfrac{1}{2}e(A,B)x_{u^*}.
\end{align*}
Combining these bounds yields the key inequality:
\[
\Sigma_1 + \Sigma_2 + \Sigma_3 \leq \left(\tfrac{1}{2}|A| + 2e(A) + e(A,B)\right)x_{u^*}.
\]
Substituting into the main equation~(\ref{eq-3}) produces:
\begin{align*}
q^2x_{u^*} &\leq \left(|A|^2 + m + \tfrac{3}{2}|A| + 2e(A) + e(A,B)\right)x_{u^*} \notag \\
&\leq (|A|^2 + 3m)x_{u^*}.
\end{align*}

Under our assumption $|A| \leq m - k - 2$ and the given lower bound $q > m - k + \tfrac{1}{m^2}$, we derive that
$
\Big(m - k + \tfrac{1}{m^2}\Big)^2 \leq (m - k - 2)^2 + 3m.
$
Simplification yields the polynomial inequality:
\begin{equation}\label{eq:final_poly}
m^5 - 4(k + 1)m^4 + 2m^3 - 2km^2 + 1 \leq 0.
\end{equation}

However, when $m \geq 7k + 31$, the left-hand side of~\eqref{eq:final_poly} is strictly positive, creating the required contradiction. Thus, 
the original assumption is false, and we conclude $|A| = m - k - 1$ by combining with~\eqref{eq:upper_bound}.
\end{proof}

Building upon our previous cardinality results and structural observations, 
we define $A_1=\{u\in A|d(u)=1\}$ and $A_2=A\setminus A_1$. We now establish the following result.
\begin{clm}\label{clm-8}
$|A_1| \geq 2$.
\end{clm}

\begin{proof}
We proceed by contradiction. Suppose to the contrary that $|A_1| \leq 1$.
Since $|A| = m - k - 1$ by Claim \ref{clm:A_size}, we have
\[ |A_2| = |A| - |A_1| \geq m - k - 2. \]
From Claim \ref{clm-6}, each vertex $u \in A$ satisfies $d_A(u) \leq 1$. 
Thus the induced subgraph on $A_2$ is a matching union some isolated vertices, yielding:
\[ e(A) + e(A,B) \geq \tfrac{1}{2}(m - k - 2).\]
The total edge count satisfies:
$e(A) + e(A,B) + e(B) = k + 1$
which yields the inequality
$ e(A) \leq k + 1 $.
Combining these observations, we obtain
\[ \tfrac{1}{2}(m - k - 2) \leq k + 1. \]
Simplifying gives that
$ m \leq 3k + 4 $.
However, this contradicts our initial assumption that $m \geq \max\{7k + 31,  k^2 + 8(k +1)\}$. 
Therefore, we conclude that $|A_1| \geq 2$.
\end{proof}

\begin{clm}\label{clm-9}
For every vertex $u \in A_1$, we have $x_u = \frac{1}{q - 1} x_{u^*}$.
\end{clm}
\begin{proof}
Recall that by definition of $A_1$, each vertex $u \in A_1$ has degree $d(u) = 1$. 
From the fundamental eigenvector equation (\ref{eq-1}),
\[ (q - d(u))x_u = \sum_{v \in N(u)} x_v. \]
Since $u \in A_1 \subseteq A = N(u^*)$, and crucially $d(u) = 1$, the unique neighbor of $u$ must be the maximal degree vertex $u^*$. 
Therefore,
\[ (q - 1)x_u =\sum_{v \in N(u)} x_v = x_{u^*}, \]
which gives that
$x_u = \frac{1}{q - 1} x_{u^*}$.
This completes the proof.
\end{proof}

\begin{clm}\label{clm-10}
For every vertex $u \in A\cup B$, we have $d(u) \leq k+2$ and 
$x_u \leq \frac{k+3}{2(q-k-2)}x_{u^*}$.
\end{clm}

\begin{proof}
For any vertex $u \in A \cup B$, the handshaking lemma combined with the maximality of $d(u^*)$ gives
\begin{equation*}
d(u^*) + d(u) \leq m + 1.
\end{equation*}
By Claim \ref{clm:A_size}, we have $d(u^*)=m-k-1$, and thus
$
d(u) \leq m + 1 - (m - k-1) =  k + 2.
$
This establishes the degree bound.

From Claim~\ref{clm-6}, we recall that $x_u < \frac{1}{2}x_{u^*}$ for all $u \in A \cup B$. 
We bound the right-hand side by considering  any neighbor $v$ of $u$:
\begin{equation*}
\sum_{v \in N(u)} x_v \leq x_{u^*} + \tfrac{1}{2}(d(u) - 1)x_{u^*} = \tfrac{1}{2}(d(u) + 1)x_{u^*}.
\end{equation*}
By the eigenvector equation (\ref{eq-1}), we have
$
(q - d(u))x_u =\sum_{v \in N(u)} x_v\leq \tfrac{1}{2}(k + 3)x_{u^*}.
$
Using the inequality $q - d(u) \geq q - k - 2$  and solving for $x_u$ yields:
\begin{equation*}
x_u \leq \frac{k + 3}{2(q - k - 2)}x_{u^*}.
\end{equation*}
This completes the claim, establishing tight bounds on both the degrees and eigenvector entries for vertices in $A \cup B$.
\end{proof}

Let us partition $B$ into two subsets: 
$B_1 = \{ w \in B \mid d_B(w) = 0 \}$ and 
$B_2 = B \setminus B_1$.
Having established this partition, we will now systematically demonstrate that both $B_1$ and $B_2$ must necessarily be empty. To see this, we first analyze the properties of  $B_2$.


\begin{clm}\label{clm-11}
$B_2 = \emptyset$.
\end{clm}

\begin{proof}
Suppose for contradiction that $B_2 \neq \emptyset$. Then there exists an edge $w_1w_2 \in E(B_2)$. 
We next analyze the properties of vertices in $B_2$ to derive a contradiction.

Since $|A| = m - k - 1$, the edge count equation gives
$ e(A) + e(A,B) + e(B) = k + 1 $.
Recall that $G$ is connected.
This implies $d_B(w_1) \leq k$ for any $w_1 \in B_2$.
By Claim \ref{clm-10}, we have:
$ x_w \leq \frac{k + 3}{2(q - k - 2)} x_{u^*}$.
Therefore, the neighborhood contribution in $B$ is bounded by:
$$ \sum_{v \in N_B(w_1)} x_v \leq d_B(w_1) \cdot \frac{k + 3}{2(q - k - 2)} x_{u^*} \leq \frac{k(k + 3)}{2(q - k - 2)} x_{u^*}. $$
 
On the other hand, 
by Claim \ref{clm-6},  $d_A(w_1) \leq 1$, we have
$ \sum\limits_{v \in N_A(w_1)}\!\!\!x_v \leq \frac{k + 3}{2(q - k - 2)} x_{u^*} $.
Notice that $N(w_1)=N_A(w_1)\cup N_B(w_1)$.
Combining these bounds gives
\[ \sum_{v \in N(w_1)} x_v \leq \frac{(k + 1)(k + 3)}{2(q - k - 2)} x_{u^*}. \]
Therefore,
from the eigenvalue equation (\ref{eq-1}), we have
\[(q - d(w_1)) x_{w_1} = \sum_{v \in N(w_1)}\!\!\!x_v \leq \tfrac{(k + 1)(k + 3)}{2(q - k - 2)} x_{u^*}.\]
Combining Claim \ref{clm-10} ($d(w_1) \leq k + 2$),  it follows that
\[ x_{w_1} \leq \frac{(k + 1)(k + 3)}{2(q - k - 2)^2} x_{u^*}.\]
The same bound holds for $x_{w_2}$.
By Claim \ref{clm-8} and Claim \ref{clm-9}, there exist $u_1, u_2 \in A_1$ with:
\[ d(u_1) = d(u_2) = 1 \quad \text{and} \quad x_{u_1} = x_{u_2} = \tfrac{1}{q-1}x_{u^*}. \]

Consider the graph $G' = G - \{w_1w_2\} + \{u_1u_2\}$, which maintains size $m$ and avoids both $K_{1,m-k}$ and $C_4$. The spectral change is:
\begin{align*}
q(G') - q(G) &\geq (x_{u_1} + x_{u_2})^2 - (x_{w_1} + x_{w_2})^2 \\
&\geq \tfrac{4}{(q-1)^2}x_{u^*}^2 - \tfrac{(k+1)^2(k+3)^2}{(q-k-2)^4}x_{u^*}^2 \\
&= \left(\tfrac{4}{(q-1)^2} - \tfrac{(k+1)^2(k+3)^2}{(q-k-2)^4}\right)x_{u^*}^2 > 0,
\end{align*}
when $m \geq \tfrac{1}{2}k^2 + 5k + \tfrac{11}{2}$. This contradicts the maximality of $q(G)$, proving $B_2 = \emptyset$.
\end{proof}

\begin{clm}\label{clm-12}
$B = \emptyset$.
\end{clm}

\begin{proof}
Assume for contradiction that $B \neq \emptyset$. By Claims~\ref{clm-11} and~\ref{clm-1}, we have:
$B = B_1 = \{ w \in B \mid d_B(w) = 0 \}$ and 
for every $w \in B$, there must exist $u \in A_2$ with $uw \in E(G)$. By Claim \ref{clm-8}, there exist 
$u_1 \neq u_2 \in A_1$ with $d(u_i)=1$ and $x_{u_i}=\frac{1}{q-1}x_{u^*}$.

For any $w \in B$ and its neighbor $u \in A_2$, we have
$x_w \leq \frac{(k+1)(k+3)}{2(q-k-2)^2}x_{u^*}$ and
$x_u \leq \frac{k+3}{2(q-k-2)}x_{u^*}$.
Below, we derive a new upper bound for $x_u$.
By equation (\ref{eq-1}), we obtain $$(q - d(u))x_u = \sum_{v \in N(u)} x_v \leq x_{u^*} + \frac{(d(u)-1)(k+3)}{2(q-k-2)}x_{u^*}.$$
From Claim \ref{clm-10}, we have $d(u) \leq k + 2$. Solving for $x_u$, we conclude  
$$x_u \leq \left(\frac{(k+1)(k+3)}{2(q-k-2)^2} + \frac{1}{q-k-2}\right)x_{u^*}.$$
Therefore, we have
\begin{equation}
x_w+x_u\leq\frac{q-k-2+(k+1)(k+3)}{(q-k-2)^2}x_{u^*}.\notag
\end{equation}
Consider the modified graph
$G' = G - \{ uw \} + \{ u_1u_2 \}.$
Clearly, $G'$ is a $\{K_{1,m-k}, C_4\}$-free graph of size $m$.
The singless Laplacian spectral change satisfies:
\begin{align*}
q(G') - q(G) 
&\geq (x_{u_1} + x_{u_2})^2 - (x_u + x_w)^2 \\
&= \left[ \frac{4}{(q-1)^2} - \left( \frac{q-k-2+(k+1)(k+3)}{(q-k-2)^2} \right)^2 \right] x_{u^*}^2 \\
&> 0 \quad (\text{since}\ \  m \geq k^2+8(k+1)).
\end{align*}
This contradicts the $q$-maximality of $G$, proving $B = \emptyset$.
\end{proof}
Synthesizing our preceding results, we have demonstrated that
the vertex set satisfies $|A| = m - k - 1$ and 
the set $B$ is empty ($B = \emptyset$).
From the fundamental edge-counting identity $m=|A|+ e(A)$, we immediately derive:
\[ e(A) = k + 1. \]

The $C_4$-free condition on $G$ imposes strong structural constraints. These constraints, combined with the established edge count, 
uniquely determine the graph's isomorphism class. We therefore conclude that $G$ must be isomorphic to the extremal graph $S_{m,k+1}^+$.

This completes the characterization of the maximal graph in our proof.
\end{proof}


\begin{thebibliography}{11}{\small


\bibitem{Brouwer-1}
A.E. Brouwer, W.H. Haemers, Spectra of Graphs,  \emph{Monatsh. Math}. 
\textbf{97} (1984) 283--305. \vspace{-0.2cm}


\bibitem{Cvetkovic-4}
D. Cvetkovi\'c, P. Rowlinson, S. Simi\'c, Signless Laplacians of finite graphs,  \emph{Linear Algebra Appl.}. 
 \textbf{423}(1) (2007) 155--171.\vspace{-0.2cm}

\bibitem{Cvetkovic-5}
D. Cvetkovi\'c, P. Rowlinson, S. Simi\'c, An Introduction to the Theory of Graph Spectra,  \emph{Cambridge University Press}. 
 \textbf{75} (2010) xi--364.\vspace{-0.2cm}


\bibitem{Cvetkovic-1}
D. Cvetkovi\'c, S.K. Simi\'c, Towards a spectral theory of graphs based on the signless Laplacian I,  \emph{Publ. Inst. Math}. 
(Beograd) \textbf{85}(99) (2009) 19--33.\vspace{-0.2cm}

\bibitem{Cvetkovic-2}
D. Cvetkovi\'c, S.K. Simi\'c, Towards a spectral theory of graphs based on the signless Laplacian II,  \emph{Linear Algebra Appl}. 
\textbf{432} (2010) 2257--2272.\vspace{-0.2cm}

\bibitem{Cvetkovic-3}
D. Cvetkovi\'c, S.K. Simi\'c, Towards a spectral theory of graphs based on the signless Laplacian III,  \emph{Appl. Anal. Discrete Math}. 
\textbf{4} (2010) 156--166.\vspace{-0.2cm}

\bibitem{Cvetkovic1} 
D.M. Cvetkovi\'{c}, M. Doob, H. Sachs, 
Spectra of Graphs--Theory and Application. III revised and enlarged edition, 
Johan Ambrosius Bart Velarg, Heidelberg-Leipzig, 1995.\vspace{-0.2cm}

\bibitem{Das-1}
K.C. Das, The largest eigenvalue of the signless Laplacian,  \emph{Linear Algebra Appl.} 
\textbf{394} (2004) 237--245.\vspace{-0.2cm}

\bibitem{De Freitas-1}
MAA. De Freitas, V. Nikiforov, L. Patuzzi, Maxima of the $Q$-index: forbidden 4-cycle and 5-cycle,  \emph{Electron. J. Linear Algebra} 
\textbf{26} (2013) 746--756.\vspace{-0.2cm}

\bibitem{De Lima-1}
L.S. De Lima, V. Nikiforov, C.S. Oliveira, The clique number and the smallest eigenvalue of signless Laplacian of graphs,  \emph{Discrete Math.} 
\textbf{313} (2014) 2353--2360.\vspace{-0.2cm}

\bibitem{Feng}
L.H. Feng, G.H. Yu, On three conjectures involving the signless Laplacian spectral radius of graphs,
\emph{Publ. Inst. Math.}  \textbf{85}(99) (2009) 35--38.
\vspace{-0.2cm}

\bibitem{Godsil-1}
C. Godsil, G. Royle, Algebraic Graph Theory,  \emph{GTM}. 
\textbf{207} (2001). \vspace{-0.2cm}


\bibitem{Jamakovic-1}
A. Jamakovic, P. Van Mieghem, On the robustness of complex networks by using the algebraic connectivity,
\emph{Networking}  (2008) 183--194.\vspace{-0.2cm}



\bibitem{Nosal-3}
E. Nosal, Eigenvalues of Graphs, Master’s Thesis,  \emph{University of Calgary},
(1970). \vspace{-0.2cm}

\bibitem{Nikiforov-5}
V. Nikiforov, X.Y. Yuan, Maxima of the $Q$-index: Forbidden even cycles, 
\emph{Linear Algebra Appl.}  \textbf{471} (2015) 636--653.\vspace{-0.2cm}

\bibitem{Nikiforov-3}
V. Nikiforov, A spectral condition for odd cycles in graphs, 
\emph{Linear Algebra Appl.}  \textbf{428} (2008) 1492--1498.\vspace{-0.2cm}

\bibitem{Nikiforov-2}
V. Nikiforov, Bounds on graph eigenvalues II, 
\emph{Linear Algebra Appl.}  \textbf{427} (2007) 183--189.\vspace{-0.2cm}

\bibitem{Nikiforov-1}
V. Nikiforov, Eigenvalues and forbidden subgraphs I, 
\emph{Linear Algebra Appl.}  \textbf{422} (2007) 284--290.\vspace{-0.2cm}

\bibitem{Nikiforov-4}
V. Nikiforov, The maximum spetral radius of $C_4$-free graphs of given order and size, 
\emph{Linear Algebra Appl.}  \textbf{430} (2009) 2898--2905.\vspace{-0.2cm}



\bibitem{Van-1}
 P. Van Mieghem, Graph Spectra for Complex Networks, 
\emph{Cambridge University Press}  (2010) xvi--346.\vspace{-0.2cm}

\bibitem{Wang-1}
Z.W. Wang, J.-M. Guo, Maximum degree and spectral radius of graphs in terms of size, 
\emph{J. Algebr. Comb.}  \textbf{59} (2024) 213--224.\vspace{-0.2cm}

\bibitem{Wang-2}
 Z.W. Wang, Generalizing theorems of Nosal and Nikiforov: triangles and quadrilaterals, 
\emph{Discrete Math.}  \textbf{345} (2022) 112973.\vspace{-0.2cm}

\bibitem{Zhai-2}
M.Q. Zhai, B. Wang, Proof of a conjecture on the spectral radius of $C_4$-free graphs,  \emph{Linear Algebra Appl}. 
\textbf{437} (2012) 1641--1647. \vspace{-0.2cm}

\bibitem{Zhai-1}
M.Q. Zhai, J. Xue, Z.Z. Lou, The signless Laplacian spectral radius of graphs with a prescribed number of edges,  \emph{Linear Algebra Appl}. 
\textbf{603} (2020) 154--165.\vspace{-0.2cm}




}
\end{thebibliography}
\end{document}